\documentclass[12pt,twoside,reqno]{amsart}

\newtheorem{thm}{Theorem}[section]
\newtheorem{lemma}{Lemma}[section]
\newtheorem{remark}{Remark}[section]
\newtheorem{proposition}{Proposition}[section]

\def \eps{\varepsilon}

\def \R{{\Bbb R}}

\numberwithin{equation}{section}

\begin{document}

\title[Stabilization of regular solutions for the ZK equation]
{
Stabilization of regular solutions for the Zakharov-Kuznetsov equation posed on bounded rectangles and on a strip
}
\author[
G.~G. Doronin,
\ N.~A. Larkin]
{
G.~G. Doronin,
\ N.~A. Larkin
\bigskip
\\
{\tiny
Departamento de Matem\'atica,\\
Universidade Estadual de Maring\'a,\\
87020-900, Maring\'a - PR, Brazil.
}
}
\address
{
Departamento de Matem\'atica\\
Universidade Estadual de Maring\'a\\
87020-900, Maring\'a - PR, Brazil.
}
\email{ggdoronin@uem.br \ \  nlarkine@uem.br}
\date{}

\subjclass
{35M20, 35Q72}
\keywords
{ZK equation, stabilization}

\begin{abstract}
Initial-boundary value problems for the 2D Zakharov-Kuznetsov equation posed on bounded rectangles and on a strip are considered.
Spectral properties of a linearized operator and critical sizes of domains are studied.
Exponential decay of regular solutions for the original nonlinear problems is proved.
\end{abstract}

\maketitle

\section{Introduction}\label{introduction}

We are concerned with initial-boundary value problems (IBVPs) posed
on bounded rectangles and on a strip located at the right half-plane
$\{(x,y)\in\mathbb{R}^2:\ x>0\}$  for the Zakharov-Kuznetsov (ZK) equation
\begin{equation}
u_t+(\alpha +u)u_x +u_{xxx}+u_{xyy}=0,\label{zk}
\end{equation}
where $\alpha $ is  equal  to $1$ or to $0,$ and which is a two-dimensional
analog of the well-known Korteweg-de Vries (KdV) equation
\begin{equation}\label{kdv}
u_t+uu_x+u_{xxx}=0
\end{equation}
with clear plasma physics applications \cite{zk}.

Equations \eqref{zk} and \eqref{kdv} are typical examples
of so-called
dispersive equations which attract considerable attention
of both pure and applied mathematicians in the past decades. The KdV
equation is probably more studied in this context.
The theory of the initial-value problem
(IVP henceforth)
for \eqref{kdv} is considerably advanced today
\cite{bona2,bourgain2,tao,kato,ponce2,kru,saut2,temam1}.

Recently, due to physics and numerics needs, publications on initial-boundary value
problems in both bounded and unbounded domains for dispersive equations have been appeared
\cite{bona1,bona3,bubnov,colin,larkin,lar2,rivas,zhang}. In
particular, it has been discovered that the KdV equation posed on a
bounded interval possesses an implicit internal dissipation. This allowed
to prove the exponential decay rate of small solutions for
\eqref{kdv} posed on bounded intervals without adding any
artificial damping term \cite{larkin}. Similar results were proved
for a wide class of dispersive equations of any odd order with one
space variable \cite{familark}.

However, \eqref{kdv} is a satisfactory approximation for real waves phenomena while the
equation is posed on the whole line ($x\in\mathbb{R}$); if
cutting-off domains are taken into account, \eqref{kdv} is no longer
expected to mirror an accurate rendition of reality. The correct
equation in this case should be written \cite{bona1,zhang}
\begin{equation}\label{1.3}
u_t+ u_x+uu_x+u_{xxx}=0.
\end{equation}
Indeed, if $x\in\R,\ t>0$,
the linear traveling term $u_x$ in \eqref{1.3} can be easily scaled
out by a simple change of variables; but it can
not be safely ignored for problems posed on both finite and semi-infinite intervals without
changes in the original domain.

Once bounded domains are considered
as a spatial region of waves propagation, their sizes appear to be
restricted by certain critical conditions. An important result
regarding these conditions is the explicit description
of a spectrum-related countable critical set
$$
\mathcal{N}=\frac{2\pi}{\sqrt3}\sqrt{k^2+kl+l^2}\,;\ \ \
k,l\in\mathbb{N}.
$$
While studying the controllability and stabilization of solutions
for \eqref{1.3}, the set $\mathcal{N}$ provides qualitative
difficulties when the length of a spatial interval coincides with
some of its elements. In fact, the function
$$
u(x)=1-\cos x
$$
is a stationary (not decaying) solution for linearized \eqref{1.3}
posed on $(0,2\pi),$ and $2\pi\in\mathcal{N}.$

It has been shown in \cite{rosier}
that control of the linear KdV equation with the term $u_x$ may fail
for critical lengths. It means that there is no
decay of solutions for a countable set of critical domains; hence, there is
no decay in a quarter-plane, at least without inclusion into
equation of some additional internal damping \cite{lipa,zuazua}. We recall, however, that
if the transport term $u_x$ is neglected, then \eqref{1.3} becomes \eqref{kdv}, and it is possible to prove the
exponential decay rate of small solutions for \eqref{kdv} posed
on any bounded interval.
More recent results on control and stabilizability for the KdV equation can
be found in \cite{rosier1,rozan}.

Quite recently, the interest on dispersive equations became to be
extended to multi-dimensional models such as Kadomtsev-Petviashvili (KP)
and ZK equations.
As far as the ZK equation is concerned,
the results on both IVP and IBVP can be found in
\cite{faminski,faminski2,pastor,pastor2,saut}. Our work has been inspired by \cite{temam} where \eqref{zk}
posed on a strip bounded in $x$ variable was considered with. Studying this paper, we
have found that the term $u_{xyy}$ in \eqref{zk} delivers additional dissipation which may ensure
decay of small solutions.
For instance,
the term $u_{xyy}$ provides the exponential decay of small solutions in
a channel-type domain; namely, in a half-strip unbounded in $x$
direction \cite{larkintronco}. However, there are restrictions on a
width of a channel in the case $\alpha=1,$ and no restrictions are needed if
$\alpha=0.$ The following questions arise naturally:
\begin{itemize}
\item
Whether width limitations for these strip-like domains are somewhat
technical?
\item Are there some critical rectangles or strips in which  solutions do not decay likewise in the KdV case?
\end{itemize}

In the present paper we put forward the hypotheses that there are critical restrictions upon the size of both
bounded and unbounded domains.
Indeed, the function
\begin{equation*}\label{1.4}
u(x,y)=\cos \left(\frac{y}{2}\right)\left(1-\cos\left(\frac{x\sqrt{3}}{2}\right)\right)
\end{equation*}
solves linearized \eqref{zk} with $\alpha=1,$ i.e., the equation
$$u_t+u_x+u_{xxx}+u_{xyy}=0,$$
considered on rectangle
$$(x,y)\in(0,4\pi/\sqrt{3})\times (-\pi,\pi),$$
and clearly it does not decay as $t\to\infty.$ This example gives raise
to expect that exact conditions (like $\mathcal{N}$ for
\eqref{1.3}) can be elaborated to describe the critical size of
domains in which the decay of solutions fails, at least for linear
models.

The main goal of our paper is to establish the existence and
uniqueness of global-in-time regular solutions of \eqref{zk} posed
both on bounded rectangles and on a strip, and the exponential decay
rate of these solutions for sufficiently small initial data.

The paper has the following structure. Section 1 is Introduction.
Section 2 contains formulation of the problem and auxiliaries.
In Section \ref{existence}, a parabolic
regularization is used to prove the existence theorem in rectangles. Uniqueness is
proved in Section 4. Existence of a
unique regular solution on a strip is established in Section 5. In Section 6, we provide
spectral arguments motivating our principal stabilization results to
be obtained in Section 7. Concerning the nonlinear ZK equation, the linear spectral
arguments seem to be technically
more difficult to apply for stabilizability than in 1D case.
Because of this, weight estimates are used in Section 7 to
prove decay of solutions instead of probably more modern unique
continuation methods \cite{coron2}.

\section{Problem and preliminaries}\label{problem}

Let $L,B,T$ be finite positive numbers. Define
\begin{equation*}
\mathcal{D}=\{(x,y)\in\mathbb{R}^2: \ x\in(0,L),\ y\in(-B,B) \},\ \ \ \mathcal{Q}_T=\mathcal{D}\times (0,T).
\end{equation*}

For $\alpha=1$ or $\alpha =0$ we consider the following IBVP:
\begin{align}
A_{\alpha}&u\equiv u_t+(\alpha +u)u_x+u_{xxx}+u_{xyy}=0,\ \ \text{in}\ \mathcal{Q}_T;
\label{2.1}
\\
&u(x,-B,t)=u(x,B,t)=0,\ \ x\in(0,L),\ t>0;
\label{2.2}
\\
&u(0,y,t)=u(L,y,t)=u_x(L,y,t)=0,\ \ y\in(-B,B),\ t>0;
\label{2.3}
\\
&u(x,y,0)=u_0(x,y),\ \ (x,y)\in\mathcal{D},
\label{2.4}
\end{align}
where $u_0:\mathcal{D}\to\mathbb{R}$ is a given function.

Hereafter subscripts $u_x,\ u_{xy},$ etc. denote the partial derivatives,
as well as $\partial_x$ or $\partial_{xy}^2$ when it is convenient.
Operators $\nabla$ and $\Delta$ are the gradient and Laplacian acting over $\mathcal{D}.$
By $(\cdot,\cdot)$ and $\|\cdot\|$ we denote the inner product and the norm in $L^2(\mathcal{D}),$
and $\|\cdot\|_{H^k}$ stands for the norm in $L^2$-based Sobolev spaces.

We will need the following result \cite{lady}.
\begin{lemma}\label{lemma1}
Let $u\in H^1(\mathcal{D})$ and $\gamma$ be the boundary of $\mathcal{D}.$

If $u|_{\gamma}=0,$ then
\begin{equation}\label{2.5}
\|u\|_{L^q(\mathcal{D})}\le \beta\|\nabla u\|^{\theta}\|u\|^{1-\theta},
\end{equation}
where $q=3$ or $q=4,$ $\theta=2\left(\frac12-\frac1q\right)$ and $\beta=2^{\theta}.$

If $u|_{\gamma}\ne0,$ then
\begin{equation}\label{2.6}
\|u\|_{L^q(\mathcal{D})}\le C_{\mathcal{D}}\|u\|^{\theta}_{H^1(\mathcal{D})}\|u\|^{1-\theta},
\end{equation}
where $C_{\mathcal{D}}$ does not depend on a size of $\mathcal{D}.$
\end{lemma}

\section{Existence theorem}\label{existence}

In this section we state the existence result for a bounded domain.
\begin{thm}\label{theorem1}
Let $\alpha=1$ and $u_0$ be a given function such that $u_0|_{\gamma}=u_{0x}|_{x=L}=0$ and
$$
I_0\equiv \|u_0\|^2_{H^1_0(\mathcal{D})}+\|\partial^2_yu_0\|^2+\|u_0u_{0x}+\Delta u_{0x}\|^2<\infty.
$$
Then for all finite positive $B,\ L,\ T$ there exists a unique regular solution to
\eqref{2.1}-\eqref{2.4} such that
\begin{align*}
&u\in L^{\infty}(0,T;H^2(\mathcal{D}))\cap L^2(0,T;H^3(\mathcal{D}));\\
&\Delta u_x\in L^{\infty}(0,T;L^2(\mathcal{D}))\cap L^2(0,T;H^1(\mathcal{D}));\\
&u_t\in L^{\infty}(0,T;L^2(\mathcal{D}))\cap L^2(0,T;H^1(\mathcal{D}))
\end{align*}
and
\begin{align}\label{33.20}
&\|u\|_{H^2(\mathcal{D})}^2(t)+\|\Delta u_x\|^2(t)+\|u_t\|^2(t)+\|u_x(0,y,t)\|^2_{H^1_0(-B,B)}\notag\\
&+\int_0^T\left\{\|u\|^2_{H^3(\mathcal{D})}(t)+\|\Delta u_x\|_{H^1(\mathcal{D})}^2(t)+\|u_x(0,y,t)\|^2_{H^2(-B,B)}\right\}\,dt\notag\\
&\le CI_0,
\end{align}
where the constant $C$ depends on $L,\ \|u_0\|$ and $T,$ but does not depend on $B>0.$
\end{thm}

To prove this theorem we consider for all real $\eps>0$
the following parabolic regularization of \eqref{2.1}-\eqref{2.4}:
\begin{align}
A^{\eps}u_{\eps}&\equiv A_1u_{\eps}+\eps(\partial_x^4u_{\eps}+\partial_y^4u_{\eps})=0\ \ \text{in}\ \mathcal{Q}_T;\label{3.1}\\
&u_{\eps}(x,-B,t)=u_{\eps}(x,B,t) \notag\\
&=\partial_y^2u_{\eps}(x,-B,t)=\partial^2_yu_{\eps}(x,B,t)=0,\ x\in(0,L),\ t>0;\label{3.2}\\
&u_{\eps}(0,y,t)=u_{\eps}(L,y,t)\notag\\
&=\partial_x^2u_{\eps}(0,y,t)=\partial_xu_{\eps}(L,y,t)=0,\ y\in (-B,B),\ t>0;\label{3.3}\\
&u_{\eps}(x,y,0)=u_{0}(x,y),\ (x,y)\in \mathcal{D}.\label{3.4}
\end{align}

For all $\eps>0,$ \eqref{3.1}-\eqref{3.4} admits, at least for small $T>0$, a unique
regular solution in $\mathcal{Q}_T$ \cite{lady2}. We assume here $u_0$ to be sufficiently smooth function
satisfying necessary compatibility conditions. Exact restrictions on $u_0$ will follow from a priori
estimates uniform in $\eps>0.$ These estimates justify passage to the limit as $\eps\to 0$ that proves the existence
part of Theorem \ref{theorem1}. Uniqueness will be studied in the sequel.

In the following subsections we are going to obtain a priori estimates independent of $\eps>0$ and $B>0.$
The subscript $\eps$ will be omitted whenever it is unambiguous.

\subsection{Estimate I}\label{1-st estimate}
Multiply \eqref{3.1} by $u_{\eps}$ and integrate over $\mathcal{D}$ to obtain
\begin{align}\label{3.5}
\|u_{\eps}\|^2(t)&+2\eps\int_0^t\left(\|\partial^2_xu_{\eps}\|^2(\tau)+\|\partial^2_yu_{\eps}\|^2(\tau)\right)\,d\tau\notag\\
&+\int_0^t\int_{-B}^Bu_{\eps x}^2(0,y,\tau)\,dy\,d\tau=\|u_0\|^2,\ \ t\in(0,T).
\end{align}

\subsection{Estimate II}\label{2-nd estimate}
Write the inner product
$$2\left(A^{\eps}u_{\eps},(1+x)u_{\eps}\right)(t)=0
$$
as
\begin{align*}
\frac{d}{dt}\left((1+x),u^2\right)(t)
&+(1-2\eps)\int_{-B}^Bu_x^2(0,y,t)\,dy\\
&+3\|u_x\|^2(t)+\|u_y\|^2(t)\\
&+2\eps\left[\left((1+x),u_{xx}^2\right)(t)+\left((1+x),u_{yy}^2\right)(t)\right]\\
&=\|u\|^2(t)+\frac23\int_{\mathcal{D}}u^3\,dx\,dy.
\end{align*}
Making use of \eqref{2.5}, we compute
\begin{align*}
\frac23\int_{\mathcal{D}}u^3\,dx\,dy
&\le \frac23\|u\|^3_{L^3(\mathcal{D})}(t)\le\frac23\left[2^{1/3}\|\nabla u\|^{1/3}(t)\|u\|^{2/3}(t)\right]^3\\
&=\frac43\|\nabla u\|(t)\|u\|^2(t)\le \delta\|\nabla u\|^2(t)+\frac{4}{9\delta}\|u\|^4(t).
\end{align*}
Taking $\eps\in(0,1/4)$ and $\delta=1/2,$ we get
\begin{align*}
\frac{d}{dt}\left((1+x),u^2\right)(t)
&+\frac12\|\nabla u\|^2(t)+\frac12\int_{-B}^Bu_x^2(0,y,t)\,dy\\
&+\eps\left(\|u_{xx}\|^2(t)+\|u_{yy}\|^2(t)\right)\le\|u\|^2(t)+\frac89\|u\|^4(t).
\end{align*}
Integration over $(0,t)$ and \eqref{3.5} then imply
\begin{align}\label{3.6}
&\left((1+x),u^2_{\eps}\right)(t)+\int_0^t\int_{-B}^Bu_{\eps x}^2(0,y,\tau)\,dy\,d\tau\notag\\
&+\int_0^t\|\nabla u_{\eps}\|^2(\tau)\,d\tau+\eps\int_0^t\left[u^2_{\eps xx}(\tau)+u^2_{\eps yy}(\tau)\right]\,d\tau\notag\\
&\le C\left((1+x),u_0^2\right),
\end{align}
where the constant $C$ does not depend on $B,\eps>0.$

\subsection{Estimate III}\label{3-d estimate}
Transforming the inner product
$$
-2\left((1+x)\partial^2_yu_{\eps},A^{\eps}u_{\eps}\right)(t)=0
$$
into the equality
\begin{align}\label{3.7}
\frac{d}{dt}
&\left((1+x),u^2_y\right)(t)
+3\|u_{xy}\|^2(t)
+(1-2\eps)\int_{-B}^{B}u_{xy}^2(0,y,t)\,dy\notag\\
&+\|u_{yy}\|^2(t)+2\eps\Bigl[\left((1+x),|\partial_y^2u_x|^2\right)(t)+\left((1+x),|\partial^3_yu|^2\right)(t)\Bigr]\notag\\
&=\|u_y\|^2(t)-2\left((1+x)uu_x,u_{yy}\right)(t),
\end{align}
we estimate
\begin{align*}
I&\equiv 2\left((1+x)uu_x,\partial^2_yu\right)(t)
=-2\left((1+x)(uu_y)_x,u_y\right)(t)\\
&=(u,u_y^2)(t)-\left((1+x)u_x,u^2_{y}\right)(t)\\
&\equiv I_1+I_2.
\end{align*}
Since ${u_y}_{\bigl|_{y=-B,B}\bigr.}\ne 0,$ we use \eqref{2.6} to estimate
\begin{align*}
I_1
&\le \|u\|(t)\|u_y\|^2_{L^4(\mathcal{D})}\le C_\mathcal{D}\|u\|(t)\|u_y\|(t)\|u_y\|_{H^1(\mathcal{D})}(t)\\
&\le \delta \|u_y\|^2_{H^1(\mathcal{D})}(t)+\frac{C^2_{\mathcal{D}}}{4\delta}\|u\|^2(t)\|u_y\|^2(t),\ \ \delta >0,
\end{align*}
and
\begin{align*}
I_2
&\le (1+L)C_{\mathcal{D}}\|u_{x}\|(t)\|u_y\|(t)\|u_y\|_{H^1(\mathcal{D})}(t)\\
&\le \delta\|u_y\|^{2}_{H^1(\mathcal{D})}(t)+\frac{1}{4\delta}(1+L)^2C^2_{\mathcal{D}}\|\nabla u\|^2(t)\left((1+x),u_y^2\right)(t).
\end{align*}
Estimates of $I_1,\ I_2$ and \eqref{3.6} give
$$
I\le 2\delta\|\nabla u_y\|^2(t)+\frac{C(L)}{\delta}\left(1+\|\nabla u\|^2(t)\right)\left((1+x),u_y^2\right)(t).
$$
Setting $\eps\in(0,1/4)$ and $\delta=1/4,$ \eqref{3.7} becomes
\begin{align}\label{3.8}
\frac{d}{dt}\left((1+x),u_y^2\right)(t)
&+\frac12\|\nabla u_y\|^2(t)+\frac12\int_{-B}^Bu_{xy}^2(0,y,t)\,dy\notag\\
&+2\eps\left(\|\partial^2_yu_x\|^2(t)+\|\partial^3_yu\|^2(t)\right)\notag\\
&\le C(L)\left(1+\|\nabla u\|^2(t)\right)\left((1+x),u_y^2\right)(t).
\end{align}
Hence, by the Gronwall lemma,
$$
\left((1+x),u_y^2\right)(t)\le CI_0,
$$
and, finally,
\begin{align}\label{3.9}
\|\partial_yu_{\eps}\|^2(t)
&+\int_0^t\|\nabla (\partial_yu_{\eps})\|^2(\tau)\,d\tau+\int_0^t\int_{-B}^B(\partial^2_{xy}u_{\eps})^2(0,y,\tau)\,dy\,d\tau\notag\\
&+\eps\int_0^t\left(\|\partial^2_y\partial_xu_{\eps}\|^2(\tau)+\|\partial^3_yu_{\eps}\|^2(\tau)\right)\,d\tau\notag\\
&\le C(L)I_0,
\end{align}
where the constant $C(L)$ depends neither on $\eps>0,$ nor on $B>0.$

\bigskip

To obtain the next estimate, we need the following simple result.
\begin{proposition}\label{prop1}
Let $u\in H^1(\mathcal{D})$ and $u_{xy}\in L^2(\mathcal{D}).$ Then
$$
\sup_{(x,y)\in\mathcal{D}}u^2(x,y,t)\le \|u\|^2_{H^1(\mathcal{D})}(t)+\|u_{xy}\|^2_{L^2(\mathcal{D})}(t).
$$
\end{proposition}
\begin{proof}
For a fixed $x\in (0,L)$ and for any $y\in (-B,B),$ it holds
$$
u^2(x,y,t)=\int_{-B}^y\partial_su^2(x,s,t)\,ds\le \int_{-B}^Bu^2(x,y,t)\,dy+\int_{-B}^Bu_y^2(x,y,t)\,dy$$
$$\equiv\rho^2(x,t).
$$
On the other hand,
$$
\sup_{(x,y)\in\mathcal{D}}u^2\le \sup_{x\in(0,L)}\rho^2(x,t)=\sup_{x\in (0,L)}\left|\int_0^x\partial_s\rho^2(s,t)\,ds\right|$$$$
\le\int_0^L\int_{-B}^B\left(u^2+u_x^2+u_y^2+u_{xy}^2\right)\,dx\,dy.
$$
The proof is complete.
\end{proof}

\subsection{Estimate IV}\label{4-th estimate}
Write
$$
2\left((1+x)\partial^4_yu_{\eps},A^{\eps}u_{\eps}\right)(t)=0
$$
in the form
\begin{align}\label{3.10}
\frac{d}{dt}\left((1+x),u_{yy}^2\right)(t)
&+(1-2\eps)\int_{-B}^Bu_{xyy}^2(0,y,t)\,dy\notag\\
&+3\|\partial^2_yu_x\|^2(t)+\|\partial^3_yu\|^2(t)\notag\\
&+2\eps\left[\left((1+x),|\partial^2_y\partial^2_xu|^2\right)(t)+\left((1+x),|\partial^4_yu|^2\right)(t)\right]\notag\\
&=\|\partial^2_yu\|^2(t)-\left((1+x)u_{yy},(u^2)_{yyx}\right)(t).
\end{align}
Denote
$$
I=-\left((1+x)u_{yy},(u^2)_{yyx}\right)(t)
=\left(u_{yy},(u^2)_{yy}\right)(t)+\left((1+x)u_{xyy},(u^2)_{yy}\right)(t)
$$
$$
\equiv I_1+I_2,
$$
where
$$
I_1=2(u_{yy},uu_{yy}+u^2_y)(t)=I_{11}+I_{12}.
$$
By Proposition \ref{prop1},
$$
I_{11}=2(u,u^2_{yy})(t)\le 2\sup_{(x,y)\in\mathcal{D}}|u(x,y,t)|\,\|u_{yy}\|^2(t)
$$
$$
\le \left(1+\|u\|^2(t)+\|\nabla u\|^2(t)+\|\nabla u_y\|^2(t)\right)\left((1+x),u^2_{yy}\right)(t)
$$
and
$$
I_{12}=2(u_{yy},u^2_y)(t)\le2\|u_{yy}\|(t)\,\|u_y\|^2_{L^4(\mathcal{D})}(t)
$$
$$
\le 2C_{\mathcal{D}}\|u_{yy}\|(t)\|u_y\|(t)\|u_y\|_{H^1(\mathcal{D})}(t)\le C\|u_y\|^2_{H^1(\mathcal{D})}(t)\|u_y\|(t).
$$
Similarly,
$$
I_2\le 2(1+L)(u_{xyy},uu_{yy}+u^2_y)(t)\le 2(1+L)\|u_{xyy}\|(t)\left(\|uu_{yy}\|(t)+\|u^2_y\|(t)\right)
$$
$$
\le \delta\|u_{xyy}\|^2(t)+\frac{2(1+L)^2}{\delta}\left(\|uu_{yy}\|^2(t)+\|u_y^2\|^2(t)\right)\le  \delta\|u_{xyy}\|^2(t)
$$$$
+\frac{2(1+L)^2}{\delta}\left[\sup_{\mathcal{D}}|u(x,y,t)|^2\left((1+x),u^2_{yy}\right)(t)
+2C_{\mathcal{D}}\|u_y\|^2(t)\|u_y\|^2_{H^1(\mathcal{D})}(t)\right].
$$
Estimates of $I_{11},\ I_{12}$ and $I_2$ then imply
$$
I\le \delta \|u_{xyy}\|^2(t)
$$
$$
+\frac{C(L)}{\delta}\Bigl[1+\|u\|^2(t)+\|\nabla u\|^2(t)+\|u_{xy}\|^2(t)\Bigr]\left((1+x),u^2_{yy}\right)(t)
$$$$
+\frac{C(L)}{\delta}\|u_y\|^2(t)\|u_y\|^2_{H^1(\mathcal{D})}(t).
$$
Inserting $I$ into \eqref{3.10}, and taking $\delta>0$ and $\eps>0$ sufficiently small, we obtain
\begin{align}\label{3.11}
\frac{d}{dt}&\left((1+x),u_{yy}^2\right)(t)+\frac12\int_{-B}^Bu^2_{xyy}(0,y,t)\,dy+\|\nabla u_{yy}\|^2(t)\notag\\
&+\eps\Bigl(\|\partial^2_x\partial^2_yu\|^2(t)+\|\partial^4_yu\|^2(t)\Bigr)\notag\\
&\le C(L)\|u_y\|^2(t)\Bigl(\|u_y\|^2(t)+\|\nabla u_y\|^2(t)\Bigr)\\
&+C(L)\Bigl[1+\|u\|^2(t)+\|\nabla u\|^2(t)+\|\nabla u_y\|^2(t)\Bigr]\left((1+x),u^2_{yy}\right)(t).\notag
\end{align}
Making use of \eqref{3.9} and the Gronwall lemma, we infer
$$
\|u_{yy}\|^2(t)\le\left((1+x),u^2_{yy}\right)(t)\le C(L)I_0.
$$
Returning to \eqref{3.11}, we conclude that
\begin{align}\label{3.12}
\|\partial^2_yu_{\eps}\|^2(t)
&+\int_0^t\|\nabla (\partial^2_yu_{\eps})\|^2(\tau)\,d\tau+\int_0^t\int_{-B}^B(\partial^3_{xyy}u_{\eps})^2(0,y,\tau)\,dy\,d\tau\notag\\
&+\eps\int_0^t\left(\|\partial^2_y\partial^2_xu_{\eps}\|^2(\tau)+\|\partial^4_yu_{\eps}\|^2(\tau)\right)\,d\tau\notag\\
&\le C(L)\left((1+x),(u_0^2+u_{0y}^2+u_{0yy}^2)\right)\le C(L)I_0
\end{align}
with $C(L)$ independent on $\eps>0,\ B>0.$

\subsection{Estimate V}\label{5-th estimate}
Write the inner product
$$
2\left((1+x)\partial_tu_{\eps},\partial_t(A^{\eps}u_{\eps})\right)(t)=0
$$
as
\begin{align}\label{3.13}
\frac{d}{dt}
&\left((1+x),u_t^2\right)(t)+(1-2\eps)\int_{-B}^Bu_{xt}^2(0,y,t)\,dy+3\|u_{xt}\|^2(t)\notag\\
&+\|u_{yt}\|^2(t)+2\eps\Bigl[\left((1+x),u_{xxt}^2\right)(t)+\left((1+x),u_{yyt}^2\right)(t)\Bigr]\notag\\
&=\|u_t\|^2+2\left((1+x)uu_t,u_{xt}\right)(t)+2(u,u_t^2)(t).
\end{align}
We calculate
\begin{align*}
I_1
&=2\left((1+x)uu_t,u_{xt}\right)(t)\\
&\le 2(1+L)^{\frac12}\|u_{xt}(t)\,\sup_{\mathcal{D}}|u(x,y,t)|\,\|(1+x)^{\frac12}u_t\|(t)\\
&\le \delta\|u_{xt}\|^2(t)+\left(\frac{1+L}{\delta}\right)\Bigl[\|u\|^2_{H^1(\mathcal{D})}(t)+\|u_{xy}\|^2(t)\Bigr]\left((1+x),u_t^2\right)(t).
\end{align*}
Analogously,
\begin{align*}
I_2
&=2(u,u_t^2)(t)\\
&\le 2\Bigl[1+\|u\|^2(t)+\|\nabla u\|^2(t)+\|u_{xy}\|^2(t)\Bigr]\left((1+x),u_t^2\right)(t).
\end{align*}
Taking $\delta>0,\ \eps>0$ sufficiently small, we transform \eqref{3.13}  into the inequality
\begin{align}\label{3.14}
\frac{d}{dt}
&\left((1+x),u_t^2\right)(t)+\frac12\int_{-B}^Bu_{xt}^2(0,y,t)\,dy+\|\nabla u_t\|^2(t)\notag\\
&+\eps\Bigl[\|\partial^2_xu_t\|^2(t)+\|\partial^2_yu_t\|^2(t)\Bigr]\notag\\
&\le C(L)\Bigl[1+\|u\|^2(t)+\|\nabla u\|^2(t)+\|u_{xy}\|^2(t)\Bigr]\left((1+x),u_t^2\right)(t).
\end{align}
By the Gronwall lemma,
$$
\left((1+x),u_t^2\right)(t)\le C(L)I_0.
$$
Therefore, \eqref{3.14} becomes
\begin{align}\label{3.15}
&\left((1+x),u_{\eps t}^2\right)(t)+\int_0^t\int_{-B}^B(\partial^2_{x\tau}u_{\eps})^2(0,y,t)\,dy\,d\tau
+\int_0^t\|\nabla \partial_{\tau}u_{\eps}\|^2(\tau)\,d\tau\notag\\
&+\eps\int_0^t\Bigl[\|\partial^2_x\partial_{\tau}u_{\eps}\|^2(\tau)+\|\partial^2_y\partial_{\tau}u_{\eps}\|^2(\tau)\Bigr]\,d\tau
\le CI_0,
\end{align}
where constant $C$ depends on $L>0,$ but does not depend on $B,\ \eps>0.$

\subsection{Estimate VI}\label{6 estimate}
From the inner product
$$2\left((1+x)A^{\eps}u_{\eps},u_{\eps}\right)(t)=0$$
we get
\begin{align}\label{3.16}
(1-2\eps)
&\int_{-B}^Bu^2_x(0,y,t)\,dy+3\|u_x\|^2(t)+\|u_y\|^2(t)\notag\\
&+2\eps \left((1+x),[u_{xx}^2+u_{yy}^2]\right)(t)\notag\\
&=\|u\|^2(t)+\frac23 \int_{\mathcal{D}}u^3\,dx\,dy-2\left((1+x)u_t,u\right)(t).
\end{align}
Acting as in Section \ref{2-nd estimate}, we find for all $\delta>0$
$$
I=\frac23\int_{\mathcal{D}}u^3\,dx\,dy\le\delta\|\nabla u\|^2(t)+\frac{C}{\delta}\|u\|^4(t),
$$
whence, taking $\delta>0,$ $\eps>0$ sufficiently small and using \eqref{3.6}, \eqref{3.15}, we reduce \eqref{3.16} to the form
\begin{equation}\label{3.17}
\int_{-B}^B(\partial_xu_{\eps})^2(0,y,t)\,dy+\|\nabla u_{\eps}\|^2(t)\le CI_0,\ \forall t\in(0,T).
\end{equation}

Now, transform
$$
-2\left((1+x)\partial_y^2u_{\eps},A^{\eps}u_{\eps}\right)(t)=0
$$
into the equality
\begin{align}\label{3.18}
(1-2\eps)
&\int_{-B}^Bu^2_{xy}(0,y,t)\,dt+3\|u_{xy}\|^2(t)+\|\partial^2_y u\|^2(t)\notag\\
&-2\left((1+x)u_t,u_{yy}\right)(t)
+2\eps\left((1+x),\left[|\partial^2_yu_x|^2+|\partial^3_yu|^2\right]\right)(t)\notag\\
&=\|u_y\|^2(t)-\left((1+x)\partial^2_{yx} (u^2),u_y\right)(t).
\end{align}
Repeating computations of Estimate \ref{3-d estimate} and taking into account \eqref{3.17}, we find out
\begin{align*}
I_1&=-\left((1+x)(u^2)_{yx},u_y\right)(t)\\
   &\le \delta\|\nabla u_y\|^2(t)+\frac{C}{\delta}\|u\|^2(t)\left(\|\nabla u\|^4(t)+\|\nabla u\|^2(t)\right),
\end{align*}
that is
$$
I_1\le \delta\|\nabla u_y\|^2(t)+\frac{C}{\delta}I_0.
$$
For $\delta,\eps>0$ sufficiently small, \eqref{3.18} reads
\begin{equation*}\label{3.19}
\|\nabla u_{\eps y}\|^2(t)+
\int_{-B}^B(\partial^2_{xy}u_{\eps})^2(0,y,t)\,dy+\eps\left[\|\partial^2_xu_{\eps}\|^2+\|\partial^2_yu_{\eps}\|^2\right](t)\le CI_0.
\end{equation*}
The constant $C$ here depends on $L,\ I_0,$ but does not depend on $B>0,$ $\eps>0.$

\bigskip

We resume Estimates I--VI as follows.
\begin{align}\label{3.20}
&\|\nabla u_{\eps}\|^2(t)+\|\nabla u_{\eps y}\|^2(t)+\|u_{\eps t}\|^2(t)+\|u_{\eps x}(0,y,t)\|^2_{H^1_0(-B,B)}\notag\\
&+\int_0^T\left\{\|\nabla u_{\eps yy}\|^2(t)+\|u_{\eps t}\|^2_{H^1(\mathcal{D})}(t)+\|u_{\eps x}(0,y,t)\|^2_{H^2(-B,B)}\right\}\,dt\notag\\
&\le C(L,T)I_0
\end{align}
and
\begin{equation}\label{3.21}
\eps \left[ \|u_{\eps xx}\|^2(t)+\|u_{\eps yy}\|^2(t)\right]\le C(L,T)I_0,
\end{equation}
where the constant $C(L,T)$ depends on $L,T,$ but does not depend on $B,\eps.$

\subsection{Passage to the limit as $\eps\to 0$}\label{limit}
It follows from \eqref{3.21} that for all $\psi\in H^2_0(\mathcal{D})$
$$
\lim_{\eps\to 0}\eps (\partial^2_xu_{\eps},\psi_{xx})(t)=0\ \ \text{ and }\ \ \lim_{\eps\to 0}\eps (\partial^2_yu_{\eps},\psi_{yy})(t)=0.
$$
Since the constants in \eqref{3.20} and \eqref{3.21} do not depend on
$\eps>0,\ B>0,$ one may pass to the limit as $\eps\to 0$ in
$$
\int_0^T\int_{\mathcal{D}}\left[\partial_tu_{\eps}+(1+u_{\eps})\partial_xu_{\eps}+\Delta\partial_xu_{\eps}\right]\psi\,dx\,dy\,dt
$$
$$
+\eps\int_0^T\int_{\mathcal{D}}\left[\partial^2_xu_{\eps}\psi_{xx}+\partial^2_yu_{yy}\psi_{yy}\right]\,dx\,dy\,dt=0
$$
to obtain
\begin{equation}\label{3.22}
\int_0^T\int_{\mathcal{D}}\left[u_t+(1+u)u_x+\Delta u_x\right]\psi\,dx\,dy\,dt=0.
\end{equation}

Thus the following assertion is true.
\begin{lemma}\label{lema1}
Let all the conditions of Theorem \ref{theorem1} hold. Then there exists a weak solution $u(x,y,t)$ to \eqref{2.1}-\eqref{2.4} such that
\begin{align}\label{3.23}
&\|\nabla u\|^2(t)+\|\nabla u_y\|^2(t)+\|u_t\|^2(t)+\|\Delta u_x\|^2(t)+\|u_{x}(0,y,t)\|^2_{H^1_0(-B,B)}\notag\\
&+\!\int_0^T\!\left\{\|\nabla u_{yy}\|^2(t)+\|\Delta u_x\|^2(t)+\|u_t\|^2_{H^1_0(\mathcal{D})}(t)+\|u_{x}(0,y,t)\|^2_{H^2(-B,B)}\right\}dt\notag\\
&\le C(L,T)I_0,\ \ \ \text{for a.e.}\ t\in(0,T),
\end{align}
where $C(L,T),$ as earlier, depends on $L,\ T,\ \|u_0\|,$ but does not depend on $B>0.$
\end{lemma}

In order to complete the proof of the existence part of Theorem \ref{theorem1}, it suffices to show that
$$
u\in L^2\left(0,T; H^3(\mathcal{D})\right),\ \ \Delta u_x\in L^2\left(0,T; H^1(\mathcal{D})\right)
$$
and
$$
u\in L^{\infty}\left(0,T; H^2(\mathcal{D})\right),\ \ \Delta u_x\in L^{\infty}\left(0,T; L^2(\mathcal{D})\right).
$$
These inclusions will be proved in the following lemmas.

\begin{lemma}\label{lema2}
A weak solution from Lemma \ref{lema1} satisfies
\begin{equation}\label{3.24}
\int_0^T\left\{\|u\|^2_{H^3(\mathcal{D})}(t)+\|\Delta u_x\|^2_{H^1(\mathcal{D})}(t)\right\}dt\le CI_0,
\end{equation}
where $C$ does not depend on $B>0.$
\end{lemma}
\begin{proof}
Taking into account \eqref{3.23} and Proposition \ref{prop1}, we write \eqref{3.22} in the form
\begin{align*}
&\Delta u_x=-u_t-(1+u)u_x\equiv f(x,y,t)\in L^2(\mathcal{Q}_T),\\
&u_x(0,y,t)\equiv \varphi(y,t)\in L^2\left(0,T;H^2(-B,B)\right),\\
&u_x(x,-B,t)=u_x(x,B,t)=u_x(L,y,t)=0.
\end{align*}
Denote $\Phi(x,y,t)=\varphi(y,t)(1-x/L)$ in $\mathcal{Q}_T.$ Obviously, $$\Phi\in L^2\left(0,T;H^2(\mathcal{D})\right).$$ Then the function
$$
v=u_x-\Phi(x,y,t)
$$
solves in $\mathcal{D}$ the elliptic problem
\begin{equation*}
\Delta v=f(x,y,t)+\Phi_{yy}(x,y,t),\ \ v|_{\gamma}=0,
\end{equation*}
which admits a unique solution $v\in L^2\left(0,T;H^2(\mathcal{D})\right)$, see \cite{lady}.
Consequently, $u_x\in L^2\left(0,T;H^2(\mathcal{D})\right).$ Therefore, \eqref{3.23} implies \eqref{3.24}.
It remains to show that the constant in \eqref{3.24} does not depend on $B>0.$
To prove this, consider the equality
$$
\int_0^T\left(\Delta v-v,\Delta v-v\right)(t)\,dt=\int_0^T\left[f-v+\Phi_{yy}\right]^2\,dt,
$$
that implies
$$
\int_0^T\left\{\|v_{xx}\|^2(t)+\|v_{yy}\|^2(t)+2\|v_{xy}\|^2(t)+2\|\nabla v\|^2(t)+\|v\|^2(t)\right\}dt\le CI_0.
$$
This gives
$$
\int_0^T\|v\|^2_{H^2(\mathcal{D})}(t)\,dt\le CI_0
$$
with $C$ independent on $B>0.$

Taking into account \eqref{3.22} and \eqref{3.23}, we complete the proof of Lemma \ref{lema2}.
\end{proof}

\begin{lemma}\label{lema3}
A weak solution given by Lemma \ref{lema1} satisfies
\begin{equation}\label{3.25}
\|u\|^2_{H^2(\mathcal{D})}(t)+\|\Delta u_x\|^2(t)\le CI_0
\end{equation}
with $C$ independent on $B>0.$
\end{lemma}
\begin{proof}
Similarly to the proof of Lemma \ref{lema2}, for almost all $t\in(0,T)$ we consider in $\mathcal{D}$ the elliptic problem
\begin{align*}
&\Delta u_x-u=-u_t-u-(1+u)u_x\equiv F(x,y,t),\\
&u_x(0,y,t)\equiv \phi (y,t),\\
&u_x(x,-B,t)=u_x(x,B,t)=u_x(L,y,t)=0.
\end{align*}
By \eqref{3.23} and Proposition \ref{prop1}, it holds
$$
F(x,y,t)\in L^{\infty}\left(0,T;L^2(\mathcal{D})\right)\ \text{ and }\ \phi (y,t)\in  L^{\infty}\left(0,T;H^1(\mathcal{D})\right).
$$
Taking $\Psi (x,y,t)=(1-x/L)\phi(y,t),$ we conclude that the function $v=u_x-\Psi$ solves the following elliptic problem:
$$
\Delta v-v=F(x,y,t)+\left(1-\frac{x}{L}\right)\partial_y\phi_y(y,t),\ \ v|_{\gamma}=0.
$$
This problem has a unique solution \cite{lady}
$$
v\in  L^{\infty}\left(0,T;H^1(\mathcal{D})\right).
$$
Taking into account \eqref{3.22} and \eqref{3.23}, we prove Lemma \ref{lema3}.
\end{proof}

In the regularization process we have
imposed suitable smoothness and consistency conditions upon $u_0$ defined actually by \eqref{3.14}. In the final steps
these excessive restrictions may be clearly weakened by usual compactness arguments.

Making use of Lemmas \ref{lema1}--\ref{lema3}, we complete the proof of the existence part of Theorem \ref{theorem1}.

\section{Uniqueness}\label{uniq}
Let $u_1$ and $u_2$ be two distinct solutions to \eqref{2.1}-\eqref{2.4}. Then $z=u_1-u_2$ solves the following IBVP:
\begin{align}
Az&\equiv z_t+\frac12(u_1^2-u_2^2)_x+\Delta z_x=0\ \ \text{in}\ \mathcal{Q}_T,\label{5.1}\\
&z(0,y,t)=z(L,y,t)=z_x(L,y,t)=z(x,\pm B,t)=0,\ \ t>0,\label{5.2}\\
&z(x,y,0)=0,\ \ (x,y)\in\mathcal{D}.\label{5.3}
\end{align}
From $$2\left(Az,(1+x)z\right)(t)=0$$ we infer
\begin{align}\label{5.4}
\frac{d}{dt}&\left((1+x),z^2\right)(t)+3\|z_x\|^2(t)+\|z_y\|^2(t)\notag\\
&+\int_{-B}^Bz_x^2(0,y,t)\,dy=-\int_{\mathcal{D}}\left[(u_1+u_2)z\right]_xz(1+x)\,dx\,dy.
\end{align}
Consider
$$
I=-\int_{\mathcal{D}}\left[(u_1+u_2)z\right]_xz(1+x)\,dx\,dy
$$
$$
=-\int_{\mathcal{D}}(u_1+u_2)z^2\,dx\,dy-\int_{\mathcal{D}}(u_1+u_2)(1+x)zz_x\,dx\,dy
$$
$$
\le \sup_{\mathcal{Q}_T}|u_1+u_2|\,\|z\|^2(t)+\|z_x\|^2(t)+(1+L)^2\sup_{\mathcal{Q}_T}|u_1+u_2|^2\|z\|^2(t).
$$
Due to \eqref{3.25} and Proposition \ref{prop1}
$$
\sup_{\mathcal{Q}_T}|u_1+u_2|^2(x,y,t)\le CI_0,
$$
whence,
$$
I\le \|z_x\|^2(t)+C\|z\|^2(t).
$$
This and \eqref{5.4} give
$$
\frac{d}{dt}\left((1+x),z^2\right)(t)\le C\left((1+x),z^2\right)(t).
$$
Gronwall's lemma and \eqref{5.3} then imply
$$
\|z\|^2(t)\equiv 0\text{ for all }t>0.
$$
The proof of uniqueness and, consequently, the proof of Theorem \ref{theorem1} is therefore completed.
\begin{flushright}$\Box$\end{flushright}

\section{Problem on a strip}\label{strip}
Taking into account that estimates of Theorem \ref{theorem1} do not depend on $B>0,$ one can expand a bounded domain
$\mathcal{D}$ to a strip
$$
\mathcal{S}_L=\left\{(x,y)\in\mathbb{R}^2:\ x\in(0,L),\ y\in\mathbb{R}\right\}.
$$
The initial boundary value problem to be considered reads
\begin{align}
A_{\alpha}&u\equiv u_t+(\alpha +u)u_x+u_{xxx}+u_{xyy}=0,\ \ \text{in}\ \mathcal{S}_L\times(0,T);
\label{55.1}
\\
&u(0,y,t)=u(L,y,t)=u_x(L,y,t)=0,\ \ y\in\mathbb{R},\ t>0;
\label{55.2}
\\
&u(x,y,0)=u_0(x,y),\ \ (x,y)\in\mathcal{S}_L.
\label{55.3}
\end{align}
The following result then holds.
\begin{thm}\label{theorem2}
Let $\alpha=1$ and $u_0$ be a given function such that
$$
\|u_0\|^2_{H^1(\mathcal{S}_L)}+\|\partial^2_yu_0\|^2_{L^2(\mathcal{S}_L)}+\|u_0u_{0x}+\Delta u_{0x}\|^2_{L^2(\mathcal{S}_L)}<\infty
$$
and
$$
u_0(0,y)=u_0(L,y)=u_{0x}(L,y)=0.
$$
Then for all finite positive numbers $L,T$ there exists a unique regular solution to
\eqref{55.1}-\eqref{55.3} such that
\begin{align*}
&u\in L^{\infty}(0,T;H^2(\mathcal{S}_L))\cap L^2(0,T;H^3(\mathcal{S}_L));\\
&\Delta u_x\in L^{\infty}(0,T;L^2(\mathcal{S}_L))\cap L^2(0,T;H^1(\mathcal{S}_L));\\
&u_t\in L^{\infty}(0,T;L^2(\mathcal{S}_L))\cap L^2(0,T;H^1(\mathcal{S}_L)).
\end{align*}
\end{thm}

\section{Spectral analysis}
In this section we provide explicit conditions defining critical sizes of bounded rectangles and unbounded strip-like domains
of $\mathbb{R}^2$ in which stabilization of solutions may not hold, at least in a linear case. Our considerations are based on
spectral-type arguments and may be viewed as motivation for posterior nonlinear studies,
as well as
a 2D generalization of the critical lengths from \cite{rosier}.

We start with the linearization of \eqref{2.1}-\eqref{2.4} with $\alpha=1$:
\begin{align}
P&u\equiv u_t+u_x+u_{xxx}+u_{xyy}=0,\ \ \text{in}\ \mathcal{Q}_T;
\label{6.1}
\\
&u(x,-B,t)=u(x,B,t)=0,\ \ x\in(0,L),\ t>0;
\label{6.2}
\\
&u(0,y,t)=u(L,y,t)=u_x(L,y,t)=0,\ \ y\in(-B,B),\ t>0;
\label{6.3}
\\
&u(x,y,0)=u_0(x,y),\ \ (x,y)\in\mathcal{D}.
\label{6.4}
\end{align}
The related eigenvalue problem for the stationary part of $P$ becomes as follows:
find $L>0,\ B>0$ and a nontrivial $v:\mathcal{D}\to\mathbb{C}$ such that
\begin{align}
&v_x+v_{xxx}+v_{xyy}=\lambda v,\ \ \lambda\in\mathbb{C},\ \ \ \text{in}\ \mathcal{D};
\label{6.5}
\\
&v(x,-B)=v(x,B)=0,\ \ x\in(0,L);
\label{6.6}
\\
&v(0,y)=v(L,y)=v_x(0,y)=v_x(L,y)=0,\ y\in(-B,B).
\label{6.7}
\end{align}
To derive \eqref{6.5}-\eqref{6.7} see, for instance, \cite{rosier,rosier1}
for the straightway approach, and \cite{coron2,glass2} for the duality arguments.
Notice that ``extra'' boundary condition $v_x(0,y)=0$ makes the
operator in \eqref{6.5}-\eqref{6.7} to be not skew-adjoint.

Separating variables as $v(x,y)=p(x)q(y),$ we infer
\begin{equation}\label{6.8}
q^{\prime\prime}+\xi q=0,\ \ q(-B)=q(B)=0
\end{equation}
and
\begin{equation}\label{6.9}
(1-\xi)p^{\prime}+p^{\prime\prime\prime}=\lambda p,\ \ p(0)=p(L)=p^\prime(0)=p^{\prime}(L)=0.
\end{equation}
Therefore,
\begin{equation*}\label{6.99}
\xi =\left(\frac{\pi n}{2B}\right)^2, \ \ n\in\mathbb{N}
\end{equation*}
and
\begin{equation*}\label{6.999}
\lambda =i\beta,\ \ \beta\in\mathbb{R}.
\end{equation*}
To find $\beta$ from \eqref{6.9}, let $\mu_j=\mu_j(B,\beta),\ j=1,2,3$ be the roots of the characteristic equation
\begin{equation}\label{6.10}
(1-\xi)\mu+\mu ^3=i\beta.
\end{equation}
Then a function
$$
p(x)=\sum_{j=1}^{3}C_jx^{k_j}e^{\mu _jx}
$$
solves the ODE in \eqref{6.9}. Here $C_j$ are constants to be determined, and $k_j$ depends on multiplicity of $\mu_j.$ Observe that
double roots of \eqref{6.10} give only $p(x)\equiv 0,$ therefore $k_j=0,\ j=1,2,3.$ Boundary conditions in \eqref{6.9} yield
$$
\sum_{j=1}^{3}C_j=0,\ \ \sum_{j=1}^{3}\mu _jC_j=0,\ \ \sum_{j=1}^{3}C_je^{\mu _jL}=0,\ \ \sum_{j=1}^{3}\mu _jC_je^{\mu _jL}=0.
$$
Solving this system, we conclude that $C_j\ne 0$ if and only if
\begin{equation}\label{6.100}
e^{\mu_jL}=e^{\mu_iL},\ \ i\ne j.
\end{equation}
Next, taking $\mu_j=is_j,$ \eqref{6.10} becomes
\begin{equation}\label{6.11}
s^3-(1-\xi)s+\beta=0.
\end{equation}
This is the real coefficients equation which always possesses at least one real root; call it $s_1\in\mathbb{R}.$
Then \eqref{6.100} implies
$$
s_2=s_1+\frac{2\pi}{L}k,\ \ s_3=s_2+\frac{2\pi}{L}l=s_1+\frac{2\pi}{L}(2k+l),\ \ k,l\in\mathbb{N}
\footnote{In general, $k,l\in\mathbb{Z} \backslash \{0\},$ but one can numerate $s_j$ to be $s_1<s_2<s_3.$}.
$$
Furthermore, Vi\`ete's formulas for \eqref{6.11} read
\begin{align*}\label{viete}
&s_1+s_2+s_3=0,\notag\\
&s_1s_2+s_1s_3+s_2s_3=-(1-\xi),\\
&s_1s_2s_3=-\beta.\notag
\end{align*}
Simple computations give
\begin{equation*}\label{6.12}
s_1=-\frac{2\pi}{3L}(2k+l),\ \ L=\frac{2\pi}{\sqrt{3}}\sqrt{\frac{k^2+kl+l^2}{1-\xi}}
\end{equation*}
and finally
\begin{equation}\label{6.13}
\left(\frac{2\pi}{L\sqrt3}\sqrt{k^2+kl+l^2}\right)^2+\left(\frac{\pi n}{2B}\right)^2=1.
\end{equation}

\begin{remark}\label{remark1}
If the size of a rectangle $\mathcal{D}=(0,L)\times(-B,B)$ satisfies \eqref{6.13}, there are
solutions to \eqref{6.1}-\eqref{6.4} which do not decay; if both $L$ and $B$ are sufficiently small, one
can expect decay (in time) of the solutions. Once either $L$ or $B$ is small, we expect decay
of solutions to problems posed on domains unbounded in one of variables; namely in a strip and/or in a half-strip.
\end{remark}

\begin{remark}\label{remark2}
If $\alpha=0,$ \eqref{6.11} reads $s^3+\beta=0$ and \eqref{6.100} fails for all $L>0.$ This means that
a decay (in time) of solutions to \eqref{6.1}-\eqref{6.4} holds for all sizes of a rectangle $\mathcal{D}.$
\end{remark}

\section{Decay of small solutions}\label{decay}
In this section, we provide sufficient conditions in order to prove the exponential decay rate of small regular
solutions to problems \eqref{2.1}-\eqref{2.4} and \eqref{55.1}-\eqref{55.3}.

We start with a bounded rectangle $\mathcal{D}=(0,L)\times (-B,B).$

\begin{thm}\label{theorem3}Let $\alpha=1$ and $B,L$ be positive real numbers such that
\begin{equation}\label{mainrestrict}
\frac{24}{L^2}+\frac{2}{B^2}-1=2A^2>0.
\end{equation}
If $$\left((1+x),u_0^2\right)<\frac{9A^4}{16(8/L^2+2/B^2)}=\frac{(3A^2LB)^2}{32(4B^2+L^2)},
$$
then regular solutions of \eqref{2.1}-\eqref{2.4} satisfy the inequality
$$
\|u\|^2(t)\le \left((1+x),u^2\right)(t)\le e^{-\frac{A^2}{1+L}t}\left((1+x),u_0^2\right).
$$
\end{thm}

First, we need the following
\begin{proposition}\label{prop2}
Let $L>0,\ B>0$ be finite numbers and $w\in H^1_0(\mathcal{D}).$ Then the following inequalities hold:
\begin{equation}\label{7.1}
\int_0^L\int_{-B}^{B}w^2(x,y)\,dx\,dy\le \frac{B^2}{2}\int_0^L\int_{-B}^{B}w^2_y(x,y)\,dx\,dy
\end{equation}
and
\begin{equation}\label{7.2}
\int_0^L\int_{-B}^{B}w^2(x,y)\,dx\,dy\le \frac{L^2}{8}\int_0^L\int_{-B}^{B}w^2_x(x,y)\,dx\,dy
\end{equation}
\end{proposition}
\begin{proof}
To prove \eqref{7.2}, consider $x\in(0,L/2).$ Then
$$
w(x,y)=\int_0^xu_{\xi}(\xi,y)\,d\xi\le x^{1/2}\left(\int_0^{L/2}w_x^2(x,y)\,dx\right)^{1/2}.
$$
Hence
$$
w^2(x,y)\le x\int_0^{L/2}w_x^2(x,y)\,dx
$$
and consequently
\begin{equation}\label{7.3}
\int_0^{L/2}\int_{-B}^{B}w^2(x,y)\,dx\,dy\le \frac{L^2}{8}\int_0^{L/2}\int_{-B}^{B}w^2_x(x,y)\,dx\,dy.
\end{equation}
Similarly,
\begin{equation*}
\int_{L/2}^L\int_{-B}^{B}w^2(x,y)\,dx\,dy\le \frac{L^2}{8}\int_{L/2}^L\int_{-B}^{B}w^2_x(x,y)\,dx\,dy.
\end{equation*}
Adding this to \eqref{7.3} gives \eqref{7.2}. Inequality \eqref{7.1} is obtained in the same manner.
Proposition \ref{prop2} is thereby proved.
\end{proof}

To prove Theorem \ref{theorem3}, consider the inner product
$$\left((1+x)A_1u,u\right)(t)=0$$ and write it as
\begin{align}\label{7.4}
\frac{d}{dt}\left((1+x),u^2\right)(t)
&+\int_{-B}^{B}u_x^2(0,y,t)\,dy+3\|u_x\|^2(t)\notag\\
&+\|u_y\|^2(t)-\|u\|^2(t)=\frac23(1,u^3)(t).
\end{align}
Making use of \eqref{2.5}, we compute
\begin{align*}
I_1
&=\frac23(1,u^3)(t)\le \frac23\left(2^{1/3}\|\nabla u\|^{1/3}(t)\|u\|^{2/3}(t)\right)^{3}\\
&\le \frac43 \|\nabla u\|(t)\|u\|^2(t)\le \delta \|u\|^2(t)+\frac{4}{9\delta} \|u\|^2(t)\|\nabla u\|^2(t)\\
&=\delta \|u\|^2(t)+\frac{4}{9\delta} \|u\|^2(t) \left(\|u_x\|^2(t)+\|u_y\|^2(t)\right)
\end{align*}
with an arbitrary $\delta>0,$ and in addition,
\begin{align*}
I_2&=3\|u_x\|^2(t)+\|u_y\|^2(t)\\
&=(3-\epsilon)\|u_x\|^2(t)+(1-\epsilon)\|u_y\|^2(t)+\epsilon\|u_x\|^2(t)+\epsilon\|u_y\|^2(t)
\end{align*}
with an arbitrary $\epsilon >0.$
By Proposition \ref{prop2}, \eqref{7.4} reduces to
\begin{align}\label{7.5}
\frac{d}{dt}
&\left((1+x),u^2\right)(t)
+\left[\frac{24}{L^2}+\frac{2}{B^2}-1-\delta -\epsilon \left(\frac{8}{L^2}+\frac{2}{B^2}\right)\right]\|u\|^2(t)\notag\\
&+\left[\epsilon-\frac{4}{9\delta}\|u\|^2(t)\right]\|u_x\|^2(t)+\left[\epsilon-\frac{4}{9\delta}\|u\|^2(t)\right]\|u_y\|^2(t)
\le 0.
\end{align}
Denote
$$2A^2=\frac{24}{L^2}+\frac{2}{B^2}-1>0$$
and take
$$
\delta=\frac{A^2}{2},\ \ \epsilon=\frac{A^2}{2\left(\frac{8}{L^2}+\frac{2}{B^2}\right)}.
$$
With this choice of $\eps$ and $\delta,$ \eqref{7.5} reads
\begin{align}\label{7.6}
\frac{d}{dt}\left((1+x),u^2\right)&(t)+A^2\|u\|^2(t)\notag\\
&+\left[\epsilon-\frac{4}{9\delta}\left((1+x),u^2\right)(t)\right]\|\nabla u\|^2(t)\le 0.
\end{align}
It is known (see, for instance, \cite{familark}), that if $\left((1+x),u_0^2\right)<9\epsilon\delta /4,$ then
$$
\left((1+x),u^2\right)(t)< \frac{9\epsilon\delta}{4}\ \text{ for all }\ t>0,
$$
and \eqref{7.6} becomes
$$
\frac{d}{dt}\left((1+x),u^2\right)(t)+ \frac{A^2}{1+L}\left((1+x),u^2\right)(t)\le 0
$$
which has a solution
$$
\|u\|^2(t)\le \left((1+x),u^2\right)(t)\le e^{-\frac{A^2}{1+L}t}\left((1+x),u_0^2\right).
$$
The proof of Theorem \ref{theorem3} is complete.
\begin{flushright}$\Box$\end{flushright}

In the case of a strip (see Section \ref{strip}), the existence result is given by Theorem \ref{theorem2}, and
for $\mathcal{S}_L=\{(x,y)\in \mathbb{R}^2:\ x\in(0,L),\ y\in\mathbb{R}\}$ the following assertion holds.

\begin{thm}\label{theorem4}
Let $\alpha=1,\ L>0$ be a finite number such that $$24/L^2-1=2A^2>0$$ and
$$
\left((1+x),u_0^2\right)<9\frac{(24-L^2)^2}{2^9L^2}.
$$
Then a regular solution to \eqref{55.1}-\eqref{55.3} satisfies
$$
\|u\|^2(t)\le \left((1+x),u^2\right)(t)\le e^{-\varrho t}\left((1+x),u_0^2\right),
$$
where $\varrho=\frac{24-L^2}{2L^2(1+L)}.$
\end{thm}

\medskip
It is clear that restrictions on $B$ and $L$ appear due to the presence of the term $u_x,$ i.e., $\alpha=1$ in \eqref{2.1}.
If $\alpha=0$, then there are no restrictions on $B>0,\ L>0$ and the following results are true:

\begin{thm}\label{theorem5}
Let $B,\ L$ be any finite positive numbers and $\alpha=0.$ If
$$
\left((1+x),u_0^2\right)<\frac{9}{32}\frac{(12B^2+L^2)^2}{L^2B^2(4B^2+L^2)},
$$
then regular solutions to \eqref{2.1}-\eqref{2.4} satisfy the inequality
$$
\|u\|^2(t)\le \left((1+x),u^2\right)(t)\le e^{-\sigma t}\left((1+x),u_0^2\right)
$$
with $\sigma=\frac{12B^2+L^2}{B^2L^2(1+L)}.$
\end{thm}

\begin{thm}\label{theorem6}
Let $L$ be any finite positive number and $\alpha=0.$ If
$$
\left((1+x),u_0^2\right)<\frac{81}{8L^2},
$$
then regular solutions to \eqref{55.1}-\eqref{55.3} satisfy the inequality
$$
\|u\|^2(t)\le \left((1+x),u^2\right)(t)\le e^{-\nu t}\left((1+x),u_0^2\right)
$$
with $\nu=\frac{12}{L^2(1+L)}.$
\end{thm}

\section*{Conclusions}
As a conclusion, we provide a comparison between conditions \eqref{6.13} and \eqref{mainrestrict}, i.e.,
a comparison between size restrictions for linear and nonlinear models.
Taking $k=l=m=1,$ \eqref{6.13} becomes
\begin{equation}\label{88.1}
\frac{4\pi^2}{L^2}+\frac{\pi^2}{4B^2}=1,
\end{equation}
and recall that \eqref{mainrestrict} reads
\begin{equation}\label{8.1}
\frac{24}{L^2}+\frac{2}{B^2}>1.
\end{equation}
Suppose $L^{\ast}>0$ and $B^{\ast}>0$ solve \eqref{88.1} and denote
$$
\mathcal{D}^{\ast}=(0,L^{\ast})\times(-B^{\ast},B^{\ast})\subset\mathbb{R}^2.
$$
Call this set the {\it minimal critical rectangle}.
If $L<L^{\ast}$ and $B<B^{\ast}$
satisfy \eqref{8.1}, then $\mathcal{D}\subset\mathcal{D}^{\ast}.$
This means that if $\mathcal{D}$ is located inside
the {\it minimal critical rectangle}, then a sufficiently small solution
to nonlinear problem \eqref{2.1}-\eqref{2.4} necessarily stabilizes.
In particular, stabilizability holds for all rectangles $\mathcal{D}$ either with the width $L<2\pi,$ or
with the height $2B<\pi.$
Furthermore, a small solution for problems posed on a sufficiently narrow strip $\mathcal{S}_L$ stabilizes as well.
Observe also that \eqref{88.1} fits well with the stabilization result from \cite{larkintronco}.



\medskip



\begin{thebibliography}{99}

\bibitem{bona2}
\newblock J. L. Bona and R. W. Smith,
\newblock The initial-value problem for the Korteweg-de Vries equation,
\newblock Phil. Trans. Royal Soc. London Series A 278 (1975), 555--601.

\bibitem{bona1}
\newblock J. L. Bona, S. M. Sun and B.-Y. Zhang,
\newblock A nonhomogeneous boundary-value problem for the Korteweg-de Vries equation posed on a finite domain,
\newblock Comm. Partial Differential Equations 28 (2003), 1391--1436.

\bibitem{bona3}
\newblock J. L. Bona, S. M. Sun and B.-Y. Zhang,
\newblock Nonhomogeneous problems for the Korteweg-de Vries and the Korteweg-de Vries-Burgers equations in a quarter plane,
\newblock Ann. Inst. H. Poincar\'{e} Anal. Non Lin\'{e}aire 25 (2008), 1145--1185.

\bibitem{bourgain2}
\newblock J. Bourgain,
\newblock On the compactness of the support of solutions of dispersive equations,
\newblock Int. Math. Res. Notices 9 (1997), 437--447.

\bibitem{bubnov}
\newblock B. A. Bubnov,
\newblock Solvability in the large of nonlinear boundary-value problems for
 the Kortewegde Vries equation in a bounded domain (Russian),
\newblock Differentsial´nye uravneniya 16 (1980), 34--41.
\newblock Engl. transl. in: Diff. Equations 16 (1980), 24--30.


\bibitem{colin}
\newblock T. Colin and J.-M. Ghidaglia,
\newblock An initial-boundary-value problem for the Korteweg-de Vries Equation posed on a finite interval,
\newblock Adv. Differential Equations 6 (2001), 1463--1492.


\bibitem{tao}
\newblock J. Colliander, M. Keel, G. Staffilani, H. Takaoka and T. Tao,
\newblock Sharp global well-posedness results for periodic and non-periodic KdV and modified KdV on R and T,
\newblock J. Amer. Math. Soc. 16 (2003), 705--749.

\bibitem{coron2}
\newblock J.-M. Coron,
\newblock Control and nonlinearity.
\newblock Mathematical Surveys and Monographs 136. American Mathematical Society, Providence,
RI, 2007. xiv+426 pp. ISBN: 978-0-8218-3668-2; 0-8218-3668-4.

\bibitem{doronin1}
\newblock G. G. Doronin and N. A. Larkin,
\newblock KdV equation in domains with moving boundaries,
\newblock J. Math. Anal. Appl. 328 (2007), 503--515.


\bibitem{faminski}
\newblock A. V. Faminskii,
\newblock The Cauchy problem for the Zakharov-Kuznetsov equation (Russian),
\newblock Differentsial'nye Uravneniya, 31 (1995), 1070--1081;
\newblock Engl. transl. in: Differential Equations 31 (1995), 1002--1012.

\bibitem{faminski2}
\newblock A. V. Faminskii,
\newblock Well-posed initial-boundary value problems for the Zakharov-Kuznetsov equation,
\newblock Electronic Journal of Differential equations 127 (2008), 1--23.

\bibitem{familark}
\newblock A. V. Faminskii and N. A. Larkin,
\newblock Initial-boundary value problems for quasilinear dispersive equations posed on a bounded interval,
\newblock Elec. J. Diff. Equations 2010 (2010), 1--20.

\bibitem{glass2}
\newblock O. Glass and S. Guerrero
\newblock Controllability of the Korteweg-de Vries equation from the right Dirichlet boundary condition,
\newblock Systems Control Lett. 59 (2010), no. 7, 390–-395.

\bibitem{kato}
\newblock T. Kato,
\newblock On the Cauchy problem for the (generalized) Korteweg-de- Vries equations,
\newblock Advances in Mathematics Suplementary Studies, Stud. Appl. Math. 8 (1983), 93--128.

\bibitem{ponce2}
\newblock C. E. Kenig, G. Ponce and L. Vega,
\newblock Well-posedness and scattering results for the generalized Korteweg-de Vries equation
 and the contraction principle,
\newblock Commun. Pure Appl. Math. 46 (1993), 527--620.

\bibitem{kru}
\newblock S. N. Kruzhkov and A. V. Faminskii,
\newblock Generalized solutions of the Cauchy problem for the Korteweg-de Vries equation,
\newblock Math. USSR Sbornik 48 (1984), 391--421.

\bibitem{lady}
\newblock O. A. Ladyzhenskaya,
\newblock The Boundary Value Problems of Mathematical Physics.
\newblock Applied Math. Sci. 49, Springer-Verlag, New York, 1985.

\bibitem{lady2}
\newblock O. A. Ladyzhenskaya, V. A. Solonnikov and N. N. Uraltseva,
\newblock Linear and Quasilinear Equations of Parabolic Type.
\newblock American Mathematical Society, Providence, Rhode Island, 1968.


\bibitem{larkin}
\newblock N. A. Larkin,
\newblock Korteweg-de Vries and Kuramoto-Sivashinsky Equations in Bounded Domains,
\newblock J. Math. Anal. Appl. 297 (2004), 169--185.

\bibitem{lar2}
\newblock N. A. Larkin, E. Tronco,
\newblock Nonlinear quarter-plane problem for the Korteweg-de Vries equation,
\newblock Electron. J. Differential Equations 2011 (2011), 1--22.


\bibitem{larkintronco}
\newblock N. A. Larkin and E. Tronco,
\newblock Decay of small solutions for the Zakharov-Kuznetsov equation posed on a half-strip,
\newblock Bol. Soc. Paran. Mat. 31 (2013), 57--64. http://www.spm.uem.br/bspm/pdf/next/Art6.pdf doi:10.5269/bspm.v31i1.15303

\bibitem{pastor}
\newblock F. Linares and A. Pastor,
\newblock Well-posedness for the 2D modified Zakharov-Kusnetsov equation,
\newblock J. Funct. Anal. 260 (2011), 1060--1085.

\bibitem{pastor2}
\newblock F. Linares, A. Pastor and J.-C. Saut,
\newblock Well-posedness for the ZK equation in a cylinder and on the background of a KdV Soliton,
\newblock Comm. Part. Diff. Equations 35 (2010), 1674--1689.


\bibitem{saut}
\newblock F. Linares and J.-C. Saut,
\newblock The Cauchy problem for the 3D Zakharov-Kuznetsov equation,
\newblock Disc. Cont. Dynamical Systems A 24 (2009), 547--565.

 \bibitem{lipa}
 \newblock F. Linares and A. F. Pazoto,
 \newblock Asymptotic behavior of the Korteweg-de Vries equation posed in a quarter plane,
 \newblock J. Differential  Equations 246 (2009), 1342--1353.

\bibitem{zuazua}
\newblock G. Perla Menzala, C. F. Vasconcellos and E. Zuazua,
\newblock Stabilization of the Korteweg-de Vries equation with localized damping,
\newblock Quart. Appl. Math. 60 (2002), 111--129.

\bibitem{rivas}
\newblock I. Rivas, M. Usman and B.-Y. Zhang,
\newblock Global
 well-posedness and asymptotic behavior of a class of
 initial-boundary value problem for the Korteweg-de Vries equation
 on a finite domain,
\newblock Math. Control Related Fields 1 (2011), 61--81.

\bibitem{rosier}
\newblock L. Rosier,
\newblock Exact boundary controllability for the Korteweg-de Vries equation on a bounded domain,
\newblock ESAIM Control Optim. Calc. Var. 2 (1997), 33--55.

\bibitem{rosier1}
\newblock L. Rosier,
\newblock A survey of controllability and stabilization results for partial differential equations,
\newblock RS - JESA 41 (2007), 365--411.

\bibitem{rozan}
\newblock L. Rosier and B.-Y. Zhang,
\newblock Control and stabilization of the KdV equation: recent progress,
\newblock J. Syst. Sci. Complexity 22 (2009), 647--682.

\bibitem{saut2}
\newblock J. C. Saut,
\newblock Sur quelques g\'{e}n\'{e}ralisations de l'\'{e}quation de Korteweg-de Vries (French),
\newblock J. Math. Pures Appl. 58 (1979), 21--61.


\bibitem{temam}
\newblock J.-C. Saut and R. Temam,
\newblock An initial boundary-value problem for the Zakharov-Kuznetsov equation,
\newblock Advances in Differential Equations 15 (2010), 1001--1031.

\bibitem{temam1}
\newblock R. Temam,
\newblock Sur un probl\`{e}me non lin\'{e}aire (French),
\newblock J. Math. Pures Appl. 48 (1969), 159--172.

\bibitem{zk}
\newblock V. E. Zakharov and E. A. Kuznetsov,
\newblock On three-dimensional solitons,
\newblock Sov. Phys. JETP 39 (1974), 285--286.


\bibitem{zhang}
\newblock B.-Y. Zhang,
\newblock Exact boundary controllability of the Korteweg-de Vries equation,
\newblock SIAM J. Control Optim. 37 (1999), 543--565.






\end{thebibliography}
\end{document}